\documentclass[11pt,reqno,tbtags]{amsart}
\usepackage{amssymb}
\usepackage{natbib}
\bibpunct[, ]{[}{]}{;}{n}{,}{,}

\numberwithin{equation}{section}
\allowdisplaybreaks

\newtheorem{theorem}{Theorem}[section]
\newtheorem{lemma}[theorem]{Lemma}

\newtheorem{corollary}[theorem]{Corollary}
\newtheorem{conjecture}[theorem]{Conjecture}

\theoremstyle{definition}
\newtheorem{example}[theorem]{Example}

\newtheorem{problem}[theorem]{Problem}
\newtheorem{remark}[theorem]{Remark}

\theoremstyle{remark}

\newenvironment{romenumerate}{\begin{enumerate}
 }{\end{enumerate}}

\newcounter{oldenumi}
\newenvironment{romenumerateq}
{\setcounter{oldenumi}{\value{enumi}}
\begin{romenumerate} \setcounter{enumi}{\value{oldenumi}}}
{\end{romenumerate}}

\newcounter{thmenumerate}

\newcounter{xenumerate}   

\newcommand{\refT}[1]{Theorem~\ref{#1}}
\newcommand{\refC}[1]{Corollary~\ref{#1}}
\newcommand{\refL}[1]{Lemma~\ref{#1}}
\newcommand{\refR}[1]{Remark~\ref{#1}}
\newcommand{\refS}[1]{Section~\ref{#1}}

\newcommand{\refP}[1]{Problem~\ref{#1}}
\newcommand{\refE}[1]{Example~\ref{#1}}

\newcommand{\refand}[2]{\ref{#1} and~\ref{#2}}

\newcommand\marginal[1]{\marginpar{\raggedright\parindent=0pt\tiny #1}}

\begingroup
  \divide\count255 by 60
  \multiply\count255 by -60
  \advance\count255 by \time
  \ifnum \count255 < 10 \xdef\klockan{\the\count1.0\the\count255}
  \else\xdef\klockan{\the\count1.\the\count255}\fi
\endgroup

\newcommand\nopf{\qed}   

\newcommand\set[1]{\ensuremath{\{#1\}}}
\newcommand\bigset[1]{\ensuremath{\bigl\{#1\bigr\}}}

\newcommand\bigpar[1]{\bigl(#1\bigr)}
\newcommand\Bigpar[1]{\Bigl(#1\Bigr)}

\def\rompar(#1){\textup(#1\textup)}    

\def\xexp(#1){e^{#1}}

\newcommand\ntoo{\ensuremath{{n\to\infty}}}

\newcommand\iid{i.i.d.\spacefactor=1000}    
\newcommand\ie{i.e.\spacefactor=1000}
\newcommand\eg{e.g.\spacefactor=1000}

\newcommand{\as}{a.s.\spacefactor=1000}
\newcommand{\aex}{a.e.\spacefactor=1000}

\newcommand{\tend}{\longrightarrow}
\newcommand\dto{\overset{\mathrm{d}}{\tend}}
\newcommand\pto{\overset{\mathrm{p}}{\tend}}
\newcommand\asto{\overset{\mathrm{a.s.}}{\tend}}

\newcommand\eqd{\overset{\mathrm{d}}{=}}

\newcommand\bbR{\mathbb R}

\newcommand\bbN{\mathbb N}  

\newcounter{CC}
\newcounter{cc}

\newcommand\E{\operatorname{\mathbb E{}}}
\renewcommand\P{\operatorname{\mathbb P{}}}

\newcommand\supp{\operatorname{supp}}

\newcommand\gd{\delta}

\newcommand\gf{\varphi}

\newcommand\gG{\Gamma}

\newcommand\gl{\lambda}

\newcommand\cE{\mathcal E}
\newcommand\cF{\mathcal F}
\newcommand\cG{\mathcal G}

\newcommand\cL{{\mathcal L}}

\newcommand\cP{\mathcal P}

\newcommand\cS{{\mathcal S}}

\newcommand\cU{{\mathcal U}}

\newcommand\ett[1]{\boldsymbol1[#1]} 

\def\[#1]{[\![#1]\!]}

\newcommand\qw{^{-1}}

\renewcommand{\=}{:=}

\newcommand\intoi{\int_0^1}

\newcommand\oi{[0,1]}
\newcommand\oiq{(0,1)}

\newcommand\dd{\,\textup{d}}

\newcommand{\tinj}{t_{\mathrm{inj}}}
\newcommand{\tind}{t_{\mathrm{ind}}}

\newcommand{\oii}{\oi^{2}}

\newcommand{\xiij}{\xi_{ij}}

\newcommand\xn{\ensuremath{[n]}}

\newcommand{\PUx}[1]{\cP_{#1}}

\newcommand{\cpq}{\overline{\cP}}

\newcommand{\cpoo}{\PUx{\infty}}

\newcommand{\cuoo}{\cU_{\infty}}
\newcommand{\cuq}{\overline{\cU}}

\newcommand\sC{\mathsf{C}}
\newcommand\bsC{\overline{\sC}}

\newcommand\sP{\mathsf{P}}
\newcommand\sK{\mathsf{K}}
\newcommand\sH{\mathsf{\mathbf2+\mathbf2}}
\newcommand\sL{\mathsf{\mathbf3+\mathbf1}}

\newcommand\ops{ordered probability space}

\newcommand\sfmuux{\ensuremath{(\cS,\cF,\allowbreak\mu,\prec)}}
\newcommand\sfmu{\ensuremath{(\cS,\cF,\mu)}}

\newcommand\oigl{\ensuremath{(\oi,\allowbreak\gl,<)}}

\newcommand\bbnn{\bbN\cup\set\infty}
\newcommand\csq{\cS^{|Q|}}

\newcommand\xQ{{|Q|}}

\newcommand\piw{\Pi_W}

\newcommand\pie{\Pi_0}
\newcommand\nn{[n]}
\newcommand\pnw{\ensuremath{P(n,W)}}

\newcommand\pnp{P(n,\Pi)}
\newcommand\pnpi{\ensuremath{P(n,\Pi)}}

\newcommand\sss{\cS}

\newcommand\precy{\prec^*}
\newcommand\precpnw{\prec_{\pnw}}

\newcommand\mul{\mu_{\mathrm L}}
\newcommand\mur{\mu_{\mathrm R}}
\newcommand\ints{\int_{\sss}}
\newcommand\wi{W_{\mathrm I}}

\newcommand\piwimu{\Pi_{\wi,\mu}}
\newcommand\pnwi{\ensuremath{P(n,\wi)}}
\newcommand\pnwimu{\ensuremath{P(n,\wi,\mu)}}
\newcommand\sssi{\sss_{\mathrm I}}
\newcommand\cgrc{\cG_{\mathrm{rc}}}
\newcommand\hmur[1]{h_{\mur}^{#1}}
\newcommand\bhmur[1]{\bh_{\mur}^{#1}}
\newcommand\Hmur[1]{H_{\mur}^{#1}}
\newcommand\bHmur[1]{\bH_{\mur}^{#1}}

\newcommand\Hmurx[1]{H_{\mur*}^{#1}}
\newcommand\bHmurx[1]{\bH_{\mur*}^{#1}}
\newcommand\td{\tilde d}
\newcommand\PP{\mathsf{P}}
\newcommand\PPL{\mathsf{P}_\mathrm{L}}
\newcommand\PPx{\PPL^*}
\newcommand\mux{\mu^*}
\newcommand\muxr{\mur^*}
\newcommand\oisupp[1]{\oiq\setminus\supp(#1)}
\newcommand\bh{\bar h}
\newcommand\bH{\bar H}
\newcommand\IO{\mathcal{IO}}
\newcommand\IG{\mathcal{IG}}
\newcommand\SO{\mathcal{SO}}
\newcommand\UIG{\mathcal{UIG}}
\newcommand\bIO{\overline{\IO}}
\newcommand\bIG{\overline{\IG}}
\newcommand\bSO{\overline{\SO}}
\newcommand\bUIG{\overline{\UIG}}
\newcommand\IOoo{\mathcal{IO}_\infty}
\newcommand\IGoo{\mathcal{IG}_\infty}
\newcommand\SOoo{\mathcal{SO}_\infty}
\newcommand\UIGoo{\mathcal{UIG}_\infty}
\newcommand\bPsi{\overline\Psi}
\newcommand\poi{\sP(\oi)}
\newcommand\refl{^\dag}
\newcommand\tW{\tilde W}

\newcommand{\Lovasz}{Lov\'asz}

\newcommand\REM[1]{{\raggedright\texttt{[#1]}\par\marginal{XXX}}}

\newcommand\citex[1]{\texttt{[#1]}}
\newcommand\refx[1]{\texttt{#1}}

\newcommand\urladdrx[1]{{\urladdr{\def~{{\tiny$\sim$}}#1}}}

\begin{document}
\title
{Limits of interval orders and semiorders}

\date{7 April 2011} 

\address{Department of Mathematics, Uppsala University, PO Box 480,
SE-751~06 Uppsala, Sweden}
\email{svante.janson@math.uu.se}
\urladdrx{http://www.math.uu.se/~svante/}


\begin{abstract} 
We study poset limits given by sequences of finite interval orders or, as a
special case, finite semiorders. In the interval order case, we show that
every such limit can be represented by a probability measure on the space of 
closed subintervals of $\oi$, and we define a subset of such measures that
yield a unique representation. In the semiorder case, we similarly find
unique representations by a class of distribution functions.
\end{abstract}

\maketitle

\section{Introduction and main results}\label{Sintro}

The theory of graph limits was founded by
by \citet{LSz} and
\citet{BCLSV1,BCLSV2}, and further developed in a series of papers by
these and other 
authors.
An analogous theory for poset limits was initiated by
\citet{BrightwellG} and further developed by \citet{SJ224}.
The purpose of the present paper is to study the special cases of limits of
interval orders and semiorders. (Cf.\ the related study of interval
graph limits in \cite{SJ254}.) 

Definitions of these classes of posets and our main results are given in
Sections \refand{Sinterval}{Ssemi}, after some preliminaries.
We show there that every interval order limit can be represented by a
probability measure on the space $\set{[a,b]:0\le a\le b\le1}$
of closed subintervals of $\oi$, and that 
every semiorder limit can be represented by a weakly increasing function 
$g:\oi\to\oi$ such that $g(x)\ge x$.
Moreover, unlike most previously studied cases of similar representations of
graph limits or poset limits, in these two cases we find explicit classes of
such 
measures and functions that yield \emph{unique} representations. 
For semiorders, 
this leads to necessary and sufficient conditions for a sequence of
semiorders to converge to a semiorder limit; these conditions use the
distributions of the numbers of predecessors or successors of points in the
semiorders. 

In \refS{Sgraphs} we discuss the connections to graph limits, including some
open problems.

\section{Preliminaries}\label{Sprel}

We assume that the reader is familiar with the theory of graph limits and the
poset version of it in \cite{SJ224}. We use the same notations as there (see
also the graph case in \cite{SJ209});
for convenience, we repeat the main definitions. 

All posets 
are assumed to be non-empty. They
are usually finite, but we will sometimes use infinite posets as well.
If $(P,<)$ is a poset, we call $P$ its \emph{ground set}.
For simplicity, we  use the same notation for a poset and
its ground set when there is no danger of confusion. 
We let $\cP$ denote the set of unlabelled finite posets.

We may regard a poset $(P,<)$ as a digraph, with vertex set $P$ and a
directed edge $i\to j$ if and only if $i<j$ for all $i,j\in P$. 
(But note that not every digraph is a poset.)

The
functional $t(Q,P)$ is defined for finite posets $P$ and $Q$
as the proportion of all maps
$\gf:Q\to P$ that are poset homomorphisms,
\ie, such that $x<_Qy \implies \gf(x)<_P\gf(y)$.
We similarly also define $\tinj(Q,P)$ as
the proportion of all injective maps $Q\to P$ that are poset
homomorphisms 
and $\tind(Q,P)$ as
the proportion of all injective maps
$\gf:Q\to P$ such that $x<_Qy \iff \gf(x)<_P\gf(y)$
(\ie, $\gf$ is an isomorphism onto an induced subposet of $P$);
equivalently, $\tind(Q,P)$ is the probability that a random labelled induced
subposet of $|Q|$ points in $P$ is isomorphic to $Q$ (for any fixed
labelling of $Q$). (If $|Q|>|P|$, we define $\tinj(Q,P)=\tind(Q,P)=0$.)

We say that a sequence $(P_n)$ of
finite posets with $|P_n|\to\infty$ \emph{converges}, if $t(Q,P_n)$
converges for every finite poset $Q$.
(All unspecified limits in this paper are as  \ntoo.)
It is easy to see that this is equivalent to convergence of 
$\tinj(Q,P_n)$ for every $Q$, or of
$\tind(Q,P_n)$ for every  $Q$.

The (discrete) space $\cP$ of finite posets can be embedded as an open dense
subspace of a compact metric space $\cpq$, such that a sequence $P_n$ with
$|P_n|\to\infty$  converges in the sense just given if and only if it
converges in the metric space $\cpq$. The space $\cpoo\=\cpq\setminus\cP$ is
the space of \emph{poset limits}.

For each poset $Q$, the functionals $t(Q,\cdot)$, $\tinj(Q,\cdot)$ and
$\tind(Q,\cdot)$ extend to continuous functionals on $\cpq$. A poset
limit $\Pi$ is uniquely determined by the sequence of numbers
$\set{t(Q,\Pi)}_{Q\in\cP}$, and also by  $\set{\tind(Q,\Pi)}_{Q\in\cP}$.

  An \emph{\ops} $\sfmuux$ is a probability space $\sfmu$ equipped with a
  partial order $\prec$ such that $\set{(x,y):x\prec y}$ is a measurable subset
  of $\cS\times\cS$.

A \emph{(poset) kernel} on an \ops{} $\sfmuux$ is a measurable function
$W:\cS\times\cS\to\oi$ such that, for $x,y,z\in\cS$,
\begin{align}
  \label{w1}
W(x,y)>0&\implies x\prec y,\\
\label{w2}
W(x,y)>0 \text { and } W(y,z)>0 &\implies W(x,z)=1.
\end{align}


When convenient, we may omit parts of the notation that are clear from
the context and say, \eg, that $\sss$ or
$(\sss,\mu)$ is a probability space or an \ops.

For  $n\in\bbN\=\set{1,2,\dots}$,
let $[n]\=\set{1,\dots,n}$, 
and let $[\infty]\=\bbN$. Thus $[n]$ is a set of cardinality
$n$ for all $n\in\bbN\cup\set\infty$.

Given a kernel $W$ on an \ops{} \sfmuux, we define for every
$n\in\bbnn$ a random poset $\pnw$ of cardinality $n$ by taking a
sequence $(X_i)_{i=1}^\infty$ of \iid{} points in $\cS$ with
distribution $\mu$;
given $(X_i)$, we then define $\pnw$
to be $\xn$ with the random partial order $\precy$ 
such that $i\precy j$ with probability $W(X_i,X_j)$, with 
(conditionally) independent choices for different pairs $(i,j)$.
(A convenient construction is to take auxiliary independent random variables
$\xiij\sim U(0,1)$, $i,j\in\bbN$, and then define
$i\precy j$ if and only if $\xiij<W(X_i,X_j)$.)  
We also use the notation $P(n,W,\mu)$ or $P(n,\mu)$ when we wish to
emphasize the dependence on $\mu$.

\begin{example}\label{EW1}
In this paper we are mainly interested in the case when 
$W(x,y)=\ett{x\prec y}$ on some \ops{} $(\sss,\mu)$.
(We use $\ett{\cE}$ to denote the indicator function of the event
  $\cE$, which is 1 if $\cE$ occurs and 0 otherwise.)
In this case 
$i\precpnw j\iff  X_i\prec X_j$. 
In other words, $\pnw$ then is (apart from the labelling) just the
subset \set{X_1,\dots,X_n} of $\cS$ with the induced order, provided
$X_1,\dots,X_n$ are distinct (or we regard \set{X_1,\dots,X_n} as a
multiset). 
In this case we use also the notation $P(n,\sss)$.
\end{example}

One of the main results in \cite{SJ224} is the following representation theorem,
parallel to the result for graph limits by \citet{LSz}. 

\begin{theorem}
  \label{T1}
Every kernel $W$ on an \ops{} \sfmuux{} 
defines a poset limit \/ $\Pi_W\in\cpoo$ such that the following holds.
\begin{romenumerate}
  \item\label{T1a}
$\pnw\asto\Pi_W$ as \ntoo.
  \item\label{T1b}
For every poset $Q\in\cP $,
\begin{equation}\label{t1b}
t(Q,\piw)=
t(Q,W)\=
\int_{\csq} \prod_{ij:i<_Q j} W(x_i,x_j) \dd \mu(x_1)\dots\dd\mu(x_\xQ).  
\end{equation}
\end{romenumerate}  
Moreover, every poset limit $\Pi\in\cpoo$ can be represented in this
way, \ie, $\Pi=\piw$ for some kernel $W$ on an \ops{} \sfmuux.	
\end{theorem}
We also use the notation $\Pi_{W,\mu}$ or $\Pi_\mu$ for $\piw$.
If $W$ is as in \refE{EW1}, we also write $\Pi_\sss$.

Unfortunately, the \ops{} and the kernel $W$ in \refT{T1} are not
unique (just as in the corresponding representation of graph
limits); see further \cite{BCL:unique,SJ224}.
Nevertheless, if $W$ and $W'$ are kernels on \ops{s} both representing the same
poset limit  $\Pi\in\cpoo$, then the random posets $\pnw$ and $P(n,W')$ have
the same distribution, for any $n\in\bbN\cup\set\infty$.
We may consequently define the random poset $\pnpi$ for a poset limit $\Pi$
such that $\pnpi\eqd\pnw$ for any kernel $W$ such that $\Pi_W=\Pi$.

It follows easily that
\begin{equation}
  \label{tqpind}
\P\bigpar{\pnp=Q} 
= \tind(Q,\Pi),
\end{equation}
for every (labelled) poset $Q$ on $\nn$,
and that 
the infinite random poset $P(\infty,\Pi)$ characterizes the poset limit
$\Pi$: 
$P(\infty,\Pi)\eqd P(\infty,\Pi') \iff \Pi=\Pi'$
 \cite[Theorem 1.16]{SJ224}.

If $\sss$ is a measurable space, we let $\PP(\sss)$ denote the space of
probability measures on $\sss$.
We denote the Lebesgue measure on $\oi$ by $\gl$. 

\section{Numbers of predecessors and successors}
Given a finite poset $P$ and a point $x\in P$, define 
\begin{align}\label{d+-}
d_-(x)\=|\set{y\in P:y<x}|
\quad\text{and}\quad 
d_+(x)\=|\set{y\in P:y>x}|; 
\end{align}
these are the
indegree and outdegree of $x$ in $P$ regarded as a digraph.
By taking $x$ to be a uniformly random point $X$ in $P$, we obtain random
variables $d_-(X)$ and $d_+(X)$; let $\nu_\pm=\nu_\pm(P)\in\poi$ be the
distributions of the normalized random variables $d_\pm(X)/|P|$.

For a kernel $W$ on an \ops{} $(\sss,\mu)$, we make the analogous
definitions
\begin{align}\label{W+-}
W_-(x)\=\ints W(y,x)\dd\mu(y)
\quad\text{and}\quad 
W_+(x)\=\ints W(x,y)\dd\mu(y),
\end{align}
and let $\nu_\pm(W)\in\poi$ be the distributions of the random variables
$W_\pm(X)$ 
where $X$ is a random point with distribution $\mu$.
In analogy with the degree distribution in the graph case 
\cite[Section 4]{SJ238}, we have the following continuity result; we equip
$\poi$ with the usual weak topology (\ie, convergence in distribution).

\begin{lemma}\label{Lnu+-}
The maps $\nu_\pm:\cP\to\poi$ extend (uniquely) to continuous maps
$\cpq\to\poi$.
Thus, for every  poset limit $\Pi$ there exist (unique) probability
	distributions $\nu_+(\Pi)$ and $\nu_-(\Pi)$ on $\oi$ such that:
  \begin{romenumerate}
\item \label{Lnu+-Pn}
 If $P_n\to \Pi$ for a sequence of posets  $P_n$, then
  $\nu_\pm(P_n)\to\nu_\pm(\Pi)$. 
\item \label{Lnu+-Pi}
The mappings $\Pi\mapsto\nu_\pm(\Pi)$ are continuous on $\cpoo$.
  \end{romenumerate}
Moreover,
\begin{romenumerateq}
\item \label{Lnu+-W}
If $W$ is a kernel representing $\Pi$, then $\nu_\pm(\Pi)=\nu_\pm(W)$.  
\end{romenumerateq}
\end{lemma}

\begin{proof}
If $P$ is a poset and $X$ a uniform random point in $P$, then
\begin{equation}\label{ika}
  \E \bigpar{d_\pm(X)/|P|}^k =t(Q_k^\pm,P),
\end{equation}
where $Q_k^-$ is the poset with $k+1$ points of which one dominates everyone
else and the others are incomparable, and $Q_k^+$ is $Q_k^-$ with the
opposite order. (As digraphs, these are stars with
all $k$ edges directed to [from] the centre.) Similarly, if $W$ is a kernel on
$(\sss,\mu)$ and $X$ is a random point with distribution $\mu$, then
\begin{equation}\label{erika}
  \E W_\pm(X)^k =t(Q_k^\pm,W).
\end{equation}
Since these random variables are bounded, their distributions are determined
by their moments. If $W$ and $W'$ are two kernels representing $\Pi$, we
have $t(Q_k^\pm,W)=t(Q_k^\pm,W')=t(Q_k^\pm,\Pi)$, and thus \eqref{erika}
shows that $W_\pm(X)\eqd W'_\pm(X)$, \ie{}
$\nu_\pm(W)=\nu_\pm(W')$. Consequently, we may uniquely define
$\nu_\pm(\Pi)\=\nu_\pm(W)$ when $\Pi_W=\Pi$, which is \ref{Lnu+-W}.

\ref{Lnu+-Pn} and \ref{Lnu+-Pi}
follow by \eqref{ika}, \eqref{erika} and the method of moments.
\end{proof}

\begin{remark}\label{Rreflection}
  The operation $P\mapsto P\refl$ that reflects the order of $P$ extends by
  continuity to an involution $\cpq\to\cpq$. This operation interchanges
  $d_-$ and $d_+$ by \eqref{d+-}, and hence it interchanges $\nu_+$ and
  $\nu_-$ on $\cP$ and thus on $\cpq$, 
\ie, $\nu_\pm(\Pi\refl)=\nu_\mp(\Pi)$ for $\Pi\in\cpq$.
\end{remark}

\section{Interval orders}\label{Sinterval}

A (finite) poset has an \emph{interval order} 
if it is isomorphic to a set of intervals in $\bbR$ with $I\prec J$ if and only
if $x<y$ for all $x\in I$, $y\in J$
(i.e., $I$ lies entirely to the left of $J$). 
See \citet{Fishburn} for other characterizations.
We define an \emph{interval order limit} to be a poset limit that is a
limit of a sequence of finite posets with interval orders.
We denote the set of (unlabelled) finite interval orders by $\IO\subset\cP$,
its closure in $\cpq$ by $\bIO\subset\cpq$
and the set of interval order limits by 
$\IOoo\=\bIO\setminus\IO=\bIO\cap\cpoo$.

Let $\sssi\=\set{[x,y]:0\le x\le y\le 1}$ be the set
of closed intervals in \oi, with the order $I\prec J$ just defined;
we identify
$\sssi$ with the triangle $\set{(x,y):0\le x\le y\le 1}\subset\oii$ 
with the partial order
$(x_1,y_1)\prec(x_2,y_2)$ if $y_1<x_2$.
(We use both interpretations of $\sssi$ interchangeably below, for
notational convenience.) 

Any probability measure $\mu$ on $\sssi$
thus defines a distribution of random intervals.
Let $\wi$ be the kernel on $\sssi$ given by
$\wi(\mathbf x_1,\mathbf x_2)\=\ett{\mathbf x_1\prec\mathbf x_2}$.
Then, see \refE{EW1}, the random poset $\pnwi=\pnwimu$ is the poset
defined by $n$ random intervals (i.i.d.~with distribution $\mu$) with the
order above; 
thus $\pnwimu$ has an interval order.

We have the following representation theorem.
If $\mu$ is a measure on $\sssi$, its left [right] marginal $\mul$ [$\mur$]
is the distribution of the left [right]
endpoint of a random interval with distribution $\mu$;
thus $\mul$ [$\mur$] 
is the measure on $\oi$ obtained from $\mu$ by projecting $\sssi$ onto the
first [second] coordinate.

\begin{theorem}\label{Tinterval}
For every probability measure $\mu$ on $\sssi$,
the fixed kernel $\wi$ on the \ops{} $(\sssi,\mu,\prec)$ defines an interval
order limit $\Pi_\mu=\piwimu$.
Conversely, every interval order limit may be represented
in this way for some (non-unique) probability measure $\mu$ on  $\sssi$. 
We may further require either that the left marginal $\mul=\gl$,
or that the right marginal $\mur=\gl$.
\end{theorem}
We cannot have both $\mul=\gl$ and $\mur=\gl$ except in the
trivial case when $\mu$ is concentrated on the diagonal \set{(x,x)}, and
then the limit is the trivial poset limit $\pie$, 
for which $P(n,\pie)$ always is an anti-chain.

\begin{proof}
  The poset limit $\piwimu$ is a.s.\ the limit of the interval ordered posets
  $P(n,\wi,\mu)$ and is thus an interval order limit.

For the converse, we use 
the same arguments as for interval graph limits 
in \cite[Section 6]{SJ254} (recalling that
the complement of the comparability graph of an interval order is an
interval graph, see \refS{Sgraphs}). We therefore only sketch the argument
and omit some details:

If $\Pi$ is an interval order limit, there are posets $P_n$ with interval
orders and $P_n\to\Pi$. We may represent $P_n$ by intervals
$I_{ni}=[a_{ni},b_{ni}]\subseteq\oi$ 
such that the left endpoints $a_{ni}$ are evenly spaced: $a_{ni}=i/|P_n|$.
Let $\mu_n\in\sP(\sssi)$ be the empirical distribution
$|P_n|\qw\sum_i\gd_{(a_{ni},b_{ni})}$. By considering a subsequence, we may
assume that $\mu_n\to\mu$ for some $\mu\in\sP(\sssi)$.
It then follows that $\mul=\gl$ and $\Pi=\Pi_\mu$.
(Note that the mapping $\mu\mapsto\Pi_\mu$ is not continuous, but it is
continuous at every $\mu$ such that $\mul$ and $\mur$ do not have a common
atom, and thus in particular when $\mul=\gl$ as in our case.)
\end{proof}

Since $\wi$ is fixed, we thus represent interval order limits by 
measures $\mu\in\sP(\sssi)$. 

\begin{remark}
Although it is natural to represent an interval order limit
by the kernel
$\wi$ on $(\sssi,\mu)$ as in \refT{Tinterval}, 
it is shown in \cite[Example 9.5]{SJ224} that
any interval order limit can also be represented by a kernel on
$\oigl$. (In this case, the space is fixed and the kernel varies. Note also
that not every kernel on $\oi$ defines an interval order limit.)
\end{remark}

The representation in \refT{Tinterval} is not unique, but we can refine it
to a unique representation. 
For a measure $\nu$ on $\oi$, define the mappings $h_\nu^\pm:\oi\to\oi$
by
\begin{align}
  h_\nu^-(x)&\=\sup\set{z<x:z\in\supp(\nu)}, \label{h-}\\
  h_\nu^+(x)&\=\inf\set{z> x:z\in\supp(\nu)}, \label{h+}
\end{align}
with $\sup\emptyset\=0$ and $\inf\emptyset\=1$.
Write the open set $\oisupp\nu$ as a  union
$\bigcup_{k=1}^N(a_k,b_k)$ of disjoint open intervals (with $0\le N\le
\infty$); then
\begin{align} \label{h-ab}
  h_\nu^-(x)&=
  \begin{cases}
a_k, & \text{if $a_k< x\le b_k$ for some $k$},	\\
x, & \text{otherwise};
  \end{cases}
\\  h_\nu^+(x)&=  \label{h+ab}
  \begin{cases}
b_k, & \text{if $a_k\le x< b_k$ for some $k$},	\\
x, & \text{otherwise}.
  \end{cases}
\end{align}
Consequently,
\begin{equation}
  \label{h+-}
h_\nu^+\circ h_\nu^- = h_\nu^+.
\end{equation}
We also define
\begin{align}\label{bhab}
  \bh_\nu^+(x)&=
  \begin{cases}
b_k, & \text{if $a_k< x\le b_k$ for some $k$},	\\
x, & \text{otherwise},
  \end{cases}
\end{align}
noting that $\bh_\nu^+(x)=h_\nu^+(x-)$ for $x>0$ and that 
$\bh_\nu^+(x)=h_\nu^+(x)$ $\gl$-a.e.

Define the mappings $H_\nu^\pm,\bH_\nu^+:\oii\to\oii$ by
\begin{align}
H_\nu^\pm(x,y)\=\bigpar{h_\nu^\pm(x),y},
&&&
\bH_\nu^+(x,y)\=\bigpar{\bh_\nu^+(x),y}. 
\label{H+-}
\end{align}
Further, for a measurable map $\gf:\sss_1\to\sss_2$ and a measure $\mu$ on
$\sss_1$, let $\gf_*(\mu)$ be the induced measure on $\sss_2$.

\begin{lemma}\label{Lhmu}
  If $\mu$ is a probability measure on $\sssi$, then $\Hmur-:\sssi\to\sssi$
and $\bHmur+(x,y)\in\sssi$ for $\mu$-a.e.\ $(x,y)$.
Thus, $\Hmurx-(\mu)$ and $\bHmurx+(\mu)$ are probability measures on
$\sssi$.
If $\mul$ is continuous, then further $\Hmur+(x,y)=\bHmur+(x,y)\in\sssi$ for
$\mu$-\aex{} $(x,y)$, and thus $\Hmurx+(\mu)=\bHmurx+(\mu)\in\cP(\sssi)$.
\end{lemma}
\begin{proof}
The result for $\Hmur-$ is obvious, since by \eqref{h-ab},
$0\le h_\nu^-(x)\le x$.  

For $\bHmur+$, let 
\begin{equation*}
 E\=\bigset{(x,y)\in\sssi:\bHmur+(x,y)\notin\sssi} 
=\bigset{(x,y): 0\le x\le y<\bhmur+(x)}.
\end{equation*}
Write, as above, 
$\oisupp{\mur}=\bigcup_k(a_k,b_k)$ with disjoint intervals $(a_k,b_k)$. 
Then $x\le y<\bhmur+(x)$ implies, by \eqref{bhab}, that
$a_k<x\le y<b_k$ for some $k$; in particular,
$y\in(a_k,b_k)\subseteq\oisupp{\mur}$. Consequently,
$\mu(E)\le \mur(\oisupp{\mur})=0$.

Finally, if $\mul$ is continuous, then $\bh_\nu^+=h_\nu^+$ $\mul$-a.e., for
any $\nu$, and thus $\bH_\nu^+=H_\nu^+$ $\mu$-\aex{}; we choose $\nu=\mur$.
\end{proof}

\begin{lemma}\label{Lmu+-}
  Let $\mu$ be a probability measure on $\sssi$, and let
  $\mu_+=\bHmurx+(\mu)$. 
Then $\Pi_{\mu_+}=\Pi_\mu$.
\end{lemma}

Note that in general, the result does not hold for
$\Hmur\pm$; this is the reason for introducing $\bHmur+$. However, we are
mainly interested in the case $\mul=\gl$, and then \refL{Lhmu} shows that we
can use $\Hmur+$ instead of $\bHmur+$.

\begin{proof}
  Let, as above 
$\oisupp{\mur}=\bigcup_k(a_k,b_k)$ with disjoint intervals $(a_k,b_k)$. 

  Consider the infinite random poset $P(\infty,\mu)=
P(\infty,\wi,\mu)$. By \refE{EW1},
  this poset is constructed by taking \iid{} random intervals
  $I_i=[L_i,R_i]$ with distribution $\mu$; 
$P(\infty,\mu)$ then is $\bbN$ with the order $i\prec j\iff R_i<L_j$
  (\ie, $I_i\prec I_j$).
Further, $P(\infty,\mu_+)$ is defined similarly using the intervals 
  $I'_i=[\bhmur+(L_i),R_i]$, which have distribution $\mu_+$. 

For every $i$ and $j$, $I_i\prec I_j\implies I_i'\prec I_j'$, and the
converse holds too except in the case
$L_j\le R_i<\bhmur+(L_j)$. By \eqref{bhab}, if this exceptional case holds,
then, for some $k$,
$a_k<L_j\le R_i<b_k$; hence $R_i\in(a_k,b_k)\subseteq\oisupp{\mur}$.
However, \as{} $R_i\in\supp(\mur)$ for all $i$, and thus
$I_i\prec I_j\iff I_i'\prec I_j'$ for all $i,j$, so
$P(\infty,\mu_+)=P(\infty,\mu)$. In other words, 
$P(\infty,\Pi_{\mu_+})\eqd P(\infty,\Pi_{\mu})$, which is equivalent to
$\Pi_{\mu_+}=\Pi_{\mu}$, see \cite{SJ224}.
\end{proof}

Let $\PPL(\sssi)\=\set{\mu\in\PP(\sssi):\mul=\gl}$.
We choose to consider only representations by measures $\mu\in\PPL(\sssi)$
in \refT{Tinterval}; the theorem then says that
$\mu\mapsto\Pi_\mu$ maps $\PPL(\sssi)$ onto $\IOoo$.
We have the following characterisation of
when two measures in $\PPL(\sssi)$ represent the same poset limit.

\begin{theorem}\label{Tequiv}
  Let $\mu$ and $\mu'$ be two measures on $\sssi$ such that
  $\mul=\mul'=\gl$.
Then the following are equivalent:
\begin{romenumerate}
\item \label{=Pi}
$\Pi_\mu=\Pi_{\mu'}$,
\item \label{=H-}
$\Hmurx-(\mu)=H^-_{\mu_R'*}(\mu')$
\item \label{=H+}
$\Hmurx+(\mu)=H^+_{\mu_R'*}(\mu')$
\item \label{=mur,H-}
$\mur=\mur'$ and $\Hmurx-(\mu)=\Hmurx-(\mu')$,
\item \label{=mur,H+}
$\mur=\mur'$ and $\Hmurx+(\mu)=\Hmurx+(\mu')$.
\end{romenumerate}

In particular, if 
$\mul=\mul'=\gl$ and
$\supp(\mur)=\oi$, 
then $\Pi_\mu=\Pi_{\mu'}\iff \mu=\mu'$.
\end{theorem}

\begin{proof}
\ref{=Pi}$\implies$\ref{=H-}.
  Consider again the infinite random poset $P(\infty,\mu)$
constructed by \iid{} random intervals   $I_i=[L_i,R_i]$ with
distribution $\mu$; thus $P(\infty,\mu)$ has the order $i\prec j\iff R_i<L_j$.
Define, for $n,i\in\bbN$,
\begin{align}
  d_{n+}(i)&\=|\set{j\le n:j\succ i}|=|\set{j\le n: L_j>R_i}|, \\
  \td_{n-}(i)&\=|\set{j\le n:
 \exists k\neq j \text{ with } k\prec i \text{ and }k\not\prec j}| \notag
\\&\phantom: \label{tdn-}
=|\set{j\le n: \exists k\neq j: L_j\le R_k< L_i}|.
\end{align}
By the law of large numbers, as \ntoo, \as{} for every $i$,
\begin{equation}\label{tdn+lim}
 \frac{ d_{n+}(i)}n
\to \mul(R_i,1]=1-R_i.
\end{equation}
Similarly, \as{} for every $i$ and $j$, 
\begin{equation}\label{gabriel}
\exists \text{ $k\neq i,j$ with }
L_j\le R_k<L_i \iff \mur[L_j,L_i)>0. 
\end{equation}
Further, for all $x,y\in\oi$,
\begin{equation}\label{magnus}
\mur(y,x)=0\iff(y,x)\cap\supp(\mur)=\emptyset\iff \hmur-(x)\le y.  
\end{equation}
Also, since $L_j$ has the distribution $\mul=\gl$ which is continuous,
$\mur[L_j,L_i)>0\iff \mur(L_j,L_i)>0$ a.s.
It follows from \eqref{tdn-} and \eqref{gabriel}--\eqref{magnus} that \as{}
\begin{equation}
  \td_{n-}(i)=|\set{j\le n: \mur(L_j, L_i)>0}|
=|\set{j\le n: L_j<\hmur-( L_i)}|.
\end{equation}
The law of large numbers now shows that, in analogy with \eqref{tdn+lim}, \as{}
\begin{equation}\label{tdn-lim}
 \frac{\td_{n-}(i)}n
\to \mul[0,\hmur-(L_i))=\hmur-(L_i).
\end{equation}
Define 
\begin{equation}
  Y_i\=\Bigpar{\limsup_\ntoo  \frac{\td_{n-}(i)}n,
1-\limsup_\ntoo  \frac{d_{n+}(i)}n} \in\oii.
\end{equation}
We have shown in \eqref{tdn+lim} and \eqref{tdn-lim} that \as{}
$Y_i=\Hmur-(L_i,R_i)$ for every $i$. By the law of large numbers again, for
every continuous function $f$ on $\oii$, \as
\begin{equation}
  \frac1n\sum_{i=1}^nf(Y_i)
\to \E f\bigpar{\Hmur-(L_1,R_1)}
=\int f\dd \Hmurx-(\mu).
\end{equation}
Consequently, the measurable functional $ \limsup n\qw\sum_{i=1}^nf(Y_i)$
of $P(n,\mu)=\pnwimu$ is a.s.\ equal to $\int f\dd \Hmurx-(\mu)$.

The same applies to $\mu'$. If \ref{=Pi} holds so $\Pi_\mu=\Pi_{\mu'}$,
then $P(n,\mu)\eqd P(n,\mu')$, and it follows that 
$\int f\dd \Hmurx-(\mu)=\int f\dd H^-_{\mu_R'*}(\mu')$ for every continuous
$f$, and thus \ref{=H-} holds.

\ref{=H-}$\implies$\ref{=mur,H-}.
By \eqref{H+-},
$\Hmurx-(\mu)$ has the same right marginal as $\mu$, \ie{} $\mur$. Hence, 
\ref{=H-} implies $\mur=\mur'$ and \ref{=mur,H-} follows.

\ref{=H+}$\implies$\ref{=mur,H+}. Similar.

\ref{=mur,H-}$\implies$\ref{=H-} and \ref{=mur,H+}$\implies$\ref{=H+}. Trivial.

\ref{=mur,H-}$\implies$\ref{=mur,H+}.
Apply $\Hmurx+$ to both sides, noting that 
$\Hmur+\circ \Hmur-=\Hmur+$ by \eqref{H+-} and \eqref{h+-}.

\ref{=mur,H+}$\implies$\ref{=Pi}.
Immediate by \refL{Lmu+-}, since \refL{Lhmu} shows that we can use $\bHmur+$
instead of $\Hmur+$.

Finally, if $\supp(\mur)=\oi$, then $\hmur-(x)=x$ and $\Hmur-(x,y)=(x,y)$, so 
$\Hmurx-(\mu)=\mu$ and $\Hmurx-(\mu')=\mu'$; 
hence \ref{=mur,H-} implies $\mu=\mu'$.
\end{proof}

Let $\PPx(\sssi)$ be the set of $\mu\in\PPL(\sssi)$ such that if $(a,b)$ is
an open subinterval of $\oiq$ with $\mur(a,b)=0$, then the restriction of
$\mu$ to $(a,b)\times\oi$ is a product measure $\gl\times \nu_{a,b}$ for
some measure $\nu_{a,b}$ (necessarily supported on $[b,1]$).

If $\mu\in\PPL(\sssi)$ and $\oisupp{\mur}=\bigcup_k(a_k,b_k)$, with
$(a_k,b_k)$ disjoint, define the measures $\nu_k$ on $\oi$ by
$\nu_k(A)\=\mu\bigpar{(a_k,b_k)\times A}/(b_k-a_k)$,
and let $\mux$ be the measure on $\oii$ that equals $\mu$ on
$\supp(\mur)\times\oi$ and $\gl\times\nu_k$ on $(a_k,b_k)\times\oi$.

\begin{lemma}
  \label{Lmux}
Let $\mu\in\PPL(\sssi)$.
Then $\mux\in\PPx(\sssi)$ and $\mu\mapsto\mux$ is a projection
onto $\PPx(\sssi)$, \ie, $\mux=\mu$ if $\mu\in\PPx(\sssi)$.
 Further, $\muxr=\mur$ and
$\Hmurx\pm(\mu)=\Hmurx\pm(\mux)$.
Moreover, if $\mu,\mu'\in\PPL(\sssi)$, then 
$\Hmurx\pm(\mu)=H^\pm_{\mur'*}(\mu')\iff \mux=(\mu')^*$. 
\end{lemma}

\begin{proof}
  Immediate from the definitions above. (Verify $\muxr=\mur$ first.)
\end{proof}

This gives our desired unique representation of interval order limits.

\begin{theorem}\label{Tunique}
  The mapping $\mu\mapsto \Pi_\mu=\piwimu$ is a bijection of $\PPx(\sssi)$
  onto the set\/ $\IOoo$ of interval order limits.
\end{theorem}
\begin{proof}
If $\mu_1,\mu_2 \in\PPL(\sssi)$, then by \refT{Tequiv} and \refL{Lmux},
$\Pi_{\mu_1}=\Pi_{\mu_2}\iff \mux_1=\mux_2$. 
Thus, using $\mu^{**}=\mux$,  the mapping is injective
on $\PPx(\sssi)$, and by \refT{Tinterval} and \refL{Lmux} it is surjective.
\end{proof}

\begin{remark}
 This bijection is not a homeomorphism if we equip $\PPx(\sssi)$
with the usual subspace topology inherited from $\PP(\sssi)$; the correct
topology is the quotient topology given by the quotient map
$\PPL(\sssi)\to\PPx(\sssi)$ given by  $\mu\mapsto\mux$.  
\end{remark}

We can also give a different type of characterization of interval order limits.
Let $\sH$ be the poset with 4 elements \set{1,2,3,4} where the only strict
inequalities are $1<2$ and $3<4$ (as a digraph, the edge set 
$E(\sH)=\set{12,34}$).
Then a poset $P$ has an interval order if
and only if $P$ has no induced subposet isomorphic to $\sH$, \ie,
$\tind(\sH,P)=0$, 
see \cite{Fishburn}.
By the same argument as in the graph case \cite{SJ255},  
it is easily seen that the following holds.

\begin{theorem}
  \label{TintervalH}
The following are equivalent, for a poset limit $\Pi$:
\begin{romenumerate}
\item 
$\Pi$ is an interval order limit.
\item 
$\tind(\sH,\Pi)=0$. 
\item 
$P(n,\Pi)$ has a.s.\  an interval order for every $n$.  
\nopf
\end{romenumerate}
\end{theorem}

\section{Semiorders}\label{Ssemi}

Let $\sL$ be the 
poset with 4 elements \set{1,2,3,4} where $1<2<3$ but $4$ is incomparable to
the others.
A \emph{semiorder} is a partial order that does not contain any induced
subposet isomorphic to $\sH$ or $\sL$, see \citet{Fishburn}.
We define a \emph{semiorder limit} as a poset limit that is the 
limit of a sequence of finite semiordered posets. 
In particular, every semiorder is an interval order 
and thus every semiorder limit is an interval order limit.
We denote the set of (unlabelled) finite semiorders by 
$\SO\subset\IO\subset\cP$,
its closure in $\cpq$ by $\bSO\subset\bIO\subset\cpq$,
and the set of semiorder limits by 
$\SOoo\=\bSO\setminus\SO=\bSO\cap\cpoo$;
thus $\SOoo\subset\IOoo\subset\cpoo$.

The arguments in \cite{SJ255} 
show the following analogy of \refT{TintervalH}.

\begin{theorem}
  \label{TsemiorderHL}
The following are equivalent, for a poset limit $\Pi$:
\begin{romenumerate}
\item 
$\Pi$ is a semiorder limit 
\item 
$\tind(\sH,\Pi)=\tind(\sL,\Pi)=0$ 
\item 
$P(n,\Pi)$ is a.s.\  semiordered  for every $n$.
\nopf
\end{romenumerate}
\end{theorem}

\begin{example}
Let $\Pi$ be the poset limit 
corresponding to the poset $\sL$ in the sense that $t(Q,\Pi)=t(Q,\sL)$ for
every finite poset $Q$, see \cite[Example 1.12]{SJ224}.
It follows from Theorems \refand{TintervalH}{TsemiorderHL} that
$\Pi\in\IOoo$ but $\Pi\notin\SOoo$.
Consequently, the inclusion $\SOoo\subset\IOoo$ is strict.
\end{example}

\begin{example}\label{Ealmostsemiorder}
  If $\sss$ is an \ops, then the kernel $\ett{x\prec y}$ as in \refE{EW1}
  defines a poset limit $\Pi_\sss$. 
\citet{BrightwellG} say that $\sss$ is an \emph{almost-semiorder}
if, in our notation, 
$$
\P\bigpar{P(4,\sss)=\sH}=\P\bigpar{P(4,\sss)=\sL}=0;
$$
equivalently, see  \eqref{tqpind}, if 
$\tind(\sH,\Pi_\sss)=\tind(\sL,\Pi_\sss)=0$.
\refT{TsemiorderHL} thus says that $\sss$ is an almost-semiorder if and only
if $\Pi_\sss$ is a semiorder limit.
\end{example}

Since a semiorder limit is an interval order limit, it can be represented as
in \refT{Tinterval} or \refT{Tunique}; 
however, only certain measures $\mu$ are possible.
We will instead use a different (but related) representation.
We need some preparations.

Let $\cG$ be the set of functions $g:\oi\to\oi$ such that $g(x)\le g(y)$
when $x\le y$ (i.e., $g$ is weakly increasing) and $g(x)\ge x$ for all
$x\in\oi$. If $g\in\cG$, let $W_g$ be the 
kernel $W_g(x,y)=\ett{g(x)<y}$ on $(\oi,\gl,<)$, and let $\Pi_g$ be the
corresponding poset limit.
Note that 
$P(n,\Pi_g)=P(n,W_g)$ is
the interval order defined by intervals
$[X_i,g(X_i)]$, with $X_1,\dots,X_n$ independent uniform random points in $\oi$.
This is a semiorder (since $g$ is weakly increasing), and thus
it follows from \refT{TsemiorderHL} that $\Pi_g$ is a semiorder limit.

\begin{remark}\label{RSg}
  Alternatively, we may define $\sss_g$ as $\oi$ with the semiorder $x\prec
  y$ when $g(x) < y$. Thus $\sss_g$ is an \ops, and $\Pi_g$ constructed
  above equals $\Pi_{\sss_g}$ constructed in \refE{Ealmostsemiorder}.
\end{remark}

Let $\cgrc$ be the set of right-continuous functions in $\cG$.
Note that any function $g\in\cG$ can be modified at its jumps to become
right-continuous, and that this will 
\as{} not change $P(n,W_g)$, so it will 
not change $\Pi_g$.
Recall that the \emph{distribution function} of a measure $\nu$ on $\bbR$
is the right-continuous function $F(t)=F_\nu(t)\=\nu(-\infty,t]$.
\begin{lemma}\label{Lnu}
Let $g\in\cgrc$,  let $\nu_\pm=\nu_\pm(\Pi_g)=\nu_\pm(W_g)$ and let
$F_\pm(t)\=\nu_\pm[0,t]$,  $t\in\oi$,
be the distribution function of $\nu_\pm$.
Then
\begin{align}
  F_-(t)&= g(t), \label{F-}
\\
F_+(t)&
=1-\min\set{x:1-g(x)\le t}
=\max\set{y:g(1-y)\ge 1-t}, \label{F+}
\intertext{and, symmetrically,}
g(x)&  \label{lnu2}
=1-\min\set{t:1-F_+(t)\le x}
=\max\set{s:F_+(1-s)\ge 1-x}. 
\end{align}
\end{lemma}
\begin{proof}
By \eqref{W+-} and the definition of $W_g$,
\begin{equation*}
  W_{g-}(x)=\intoi\ett{g(y)<x}\dd y
=\sup\set{y:g(y)<x}, 
\end{equation*}
(with $\sup\emptyset=0$)
and thus, for $x,t\in\oi$, 
\begin{equation*}
W_{g-}(x)\le t \iff \bigpar{y>t \implies g(y)\ge x}  
\iff g(t)\ge x.
\end{equation*}
Hence,
\begin{equation*}
  F_-(t)=\gl\set{x:W_{g-}(x)\le t}
=\gl\set{x:x\le g(t)}=g(t),
\end{equation*}
showing \eqref{F-}.

Similarly,  $W_{g+}(x)=1-g(x)$,
so  $ F_+(t)=\gl\set{x:1-g(x)\le t}$.
Thus
$$
1-g(x)\le t \iff 1-F_+(t)\le x, \qquad t,x\in\oi,
$$
which implies \eqref{F+}  and \eqref{lnu2}.
\end{proof}

\begin{remark}\label{RF+-}
Thus, if we assume $g\in\cgrc$, the relation between $g=F_-$ and 
$F_+$ is symmetric. Geometrically, we obtain the graph of
$F_-$ from the graph of $F_+$ (or conversely) by reflection in the line $x+y=1$,
adjusting the result to become right-continuous.  
\end{remark}

\begin{theorem}\label{Tsemi}
  If $g\in \cG$, then $\Pi_g$ is a semiorder limit. Conversely, every
  semiorder limit equals $\Pi_g$ for a unique $g\in\cgrc$.
Thus, the mapping $g\mapsto\Pi_g$ is a bijection of $\cgrc$ onto $\SOoo$.

More precisely, if $\Pi$ is a semiorder limit,
then $\Pi=\Pi_{F_-}$ where
$F_-\in\cgrc$ is the
distribution function of $\nu_-(\Pi)$.
Alternatively, if $F_+(t)$ is the
distribution function of $\nu_+(\Pi)$, then $\Pi=\Pi_g$ where $g\in\cgrc$ is
given by \eqref{lnu2}.
\end{theorem}

\begin{proof}
  We have already remarked that $\Pi_g$ is a semiorder limit.

Conversely,
suppose that $\Pi$ is a semiorder limit.
The complement of the comparability graph of a semiorder is a unit
interval graph, 
and we argue as in
\cite[Section 10.4]{SJ254}, again omitting some details:
There exist semiorders $P_n$ with $P_n\to\Pi$.
As in the proof of \refT{Tinterval}, we can represent $P_n$ by intervals 
$[a_{ni},b_{ni}]$ with $a_{ni}=i/|P_n|$; further, since $P_n$ is a
semiorder, we may assume $b_{n1}\le b_{n2}\le\dots$.
Let $\mu_n$ be as in the proof of \refT{Tinterval}, 
and let again $\mu\in\sP(\sssi)$ be the limit of a subsequence;
thus $\Pi=\Pi_\mu$.
Moreover, now
the assumptions on $a_{ni}$ and $b_{ni}$ imply that if $x_1<x_2$ and
$y_1>y_2$, then $(x_1,y_1)$ and $(x_2,y_2)$ cannot both belong to
$\supp(\mu)$. 
Define $g$ by
\begin{equation*}
g(x)\=\inf\bigset{y:(z,y)\in\supp(\mu) \text{ for some }z>x}  
\end{equation*}
(with $\inf\emptyset=1$). Then $g\in\cgrc$, and, by the property of
$\supp(\mu)$ just shown, 
$\supp(\mu)\subseteq\set{(x,y): g(x-)\le y\le  g(x)}$, \ie, $\supp(\mu)$ is
a subset of the graph of $g$ with added vertical lines at the jumps.
Since $\mul=\gl$, and the set of jumps of $g$ is (at most) countable,
the vertical lines have measure 0 and can be ignored, and
it follows that the map $\gf:\oi\to\sssi$ given by $\gf(x)\=(x,g(x))$ is
measure preserving $(\oi,\gl)\to(\sssi,\mu)$.
(It also follows that $\supp(\mu)=\overline{\set{(x,g(x)}}$.)
Consequently, $\Pi$, which equals $\Pi_\mu$ and thus
can be represented by the kernel $\wi$ on
$(\sssi,\mu)$, can also be represented by the pullback $\wi^\gf$ on $(\oi,\gl)$;
moreover,
it is immediately seen that $\wi^\gf=W_g$. 
Thus $\Pi$ is represented by $W_g$, so $\Pi=\Pi_g$.
This shows that the mapping $g\mapsto\Pi_g$ is onto $\SOoo$.

If $\Pi=\Pi_g$ and $F_-$ is the distribution function of $\nu_-(\Pi)$, then
$g=F_-$ by \refL{Lnu},
which shows that $g\mapsto\Pi_g$ is injective, and thus a bijection
$\cgrc\to\SOoo$.

The final claims follow by \refL{Lnu}.
\end{proof}

\begin{corollary}
  \label{Csemi-}
The mapping\/  $\Pi\mapsto\nu_-(\Pi)$ is a homeomorphism of 
$\SOoo$ onto the set 
\begin{equation}\label{p-oi}
\sP_-(\oi)\=
\bigset{\nu\in\sP(\oi):\nu[0,t]\ge t \text{ for } t\in\oi}  .
\end{equation}
The inverse mapping $\sP_-(\oi)\to\SOoo$ is given by
$\nu\mapsto\Pi_\nu\=\Pi_{F_\nu}$.
\end{corollary}
\begin{proof}
  The mapping $\nu\mapsto F_\nu$ is a bijection of $\sP_-(\oi)$ onto $\cgrc$,
  and thus \refT{Tsemi} shows that
the mapping $\nu\mapsto \Pi_{F_\nu}$
is a bijection with inverse $\Pi\mapsto\nu_-(\Pi)$. Since $\nu_-$ is
continuous by \refL{Lnu+-} and the spaces are compact, the mappings are
homeomorphisms. 
\end{proof}

\begin{corollary}
  \label{Csemi+}
The mapping\/  $\Pi\mapsto\nu_+(\Pi)$ is a homeomorphism of
$\SOoo$ onto the set 
$\sP_-(\oi)$.
The inverse mapping $\sP_-(\oi)\to\SOoo$ is given by
$\nu\mapsto\Pi_\nu\refl$.
\end{corollary}
\begin{proof}
  By \refC{Csemi-} and \refR{Rreflection}.
\end{proof}

These corollaries are analogous to the results for threshold graph limits in
\cite{SJ238}, where the limits are characterized by their degree
distributions.

\begin{remark}
The definition \eqref{p-oi} says that $\sP_-(\oi)$ is the set of all
distributions on $\oi$ 
that are \emph{stochastically  smaller} than the uniform distribution.  
\end{remark}

\begin{corollary}
  Every semiorder limit equals\/ $\Pi_\sss$ for some semiordered probability
  space $\sss$.
\end{corollary}
\begin{proof}
  By \refT{Tsemi} and \refR{RSg}.
\end{proof}

Recall that the converse holds by \refE{Ealmostsemiorder}.

We proceed to corresponding limit results.
\begin{theorem}\label{Tsemilim}
Let $P_n$ be a sequence of finite semiorders with $|P_n|\to\infty$. Then
the following are equivalent.
  \begin{romenumerate}
  \item \label{tsemilim}
$P_n$ converges to a poset limit.
  \item \label{tsemilim-}
The distributions $\nu_-(P_n)$ converge.
  \item \label{tsemilim+}
The distributions $\nu_+(P_n)$ converge.
  \end{romenumerate}
  The poset limit in \ref{tsemilim} is necessarily a semiorder limit.

If $\nu_-(P_n)\to\nu_-$, then 
$P_n\to\Pi_g$, where $g=F_-$ 
is the distribution function of $\nu_-$.
Thus $g(t)=\lim_n F_{\nu_-(P_n)}(t)$ for every continuity point $t$ of $g$.

If $\nu_+(P_n)\to\nu_+$, then 
$P_n\to\Pi_g$, where $g$ is given by \eqref{lnu2} where $F_+(t)$ 
is the distribution function of $\nu_+$.
Thus $F_+(t)=\lim_n F_{\nu_+(P_n)}(t)$ for every continuity point $t$ of
$F_+$,
and we can replace $F_+$ by $\limsup_n F_{\nu_+(P_n)}(t)$ in \eqref{lnu2}.
\end{theorem}
\begin{proof}
  If \ref{tsemilim} holds, then \ref{tsemilim-} and \ref{tsemilim+}  hold by
  \refL{Lnu+-}. 

Assume now that \ref{tsemilim-} holds. 
Thus $\nu_-(P_n)\to\nu_-$ for some distribution $\nu_-\in\sP(\oi)$. 
Let $F_-$ be the
distribution function of $\nu_-$.
Suppose that $\Pi\in\cpoo$ 
is the limit of a subsequence of $P_n$. 
Then \refL{Lnu+-}
and our assumption $\nu_-(P_n)\to\nu_-$ imply that $\nu_-(\Pi)=\nu_-$.
Moreover, $\Pi$ is a semiorder limit, and thus \refT{Tsemi} implies that
$\Pi=\Pi_{F_-}$. 
  
Consequently every convergent subsequence of $P_n$ has the same limit
$\Pi_{F_-}$, and 
(since $\cpq$ is compact) this implies that the full
sequence $P_n$ converges to $\Pi_{F_-}$.
Hence \ref{tsemilim-}$\implies$\ref{tsemilim}.

\ref{tsemilim+}$\implies$\ref{tsemilim} follows by a similar argument,
or from the implication
\ref{tsemilim-}$\implies$\ref{tsemilim} by reversing the orders as in
\refR{Rreflection}. 

The final claims follow by \refL{Lnu+-} and \refT{Tsemi}, recalling that
if $\nu_n$ and $\nu$ are probability measures on $\bbR$ with distribution
functions $F_n$ and $F$, then
$\nu_n\to\nu$ if and only if $F_n(t)\to F(t)$ at every continuity point $t$
of $F$. 
\end{proof}

This theorem extends to random posets.

\begin{theorem}\label{Tsemilimrandom}
Let $P_n$ be a sequence of random finite semiorders with $|P_n|\pto\infty$. Then
the following are equivalent.
  \begin{romenumerate}
  \item \label{tsemilimrandom}
$P_n\dto\Pi$ for some random poset limit $\Pi$.
  \item \label{tsemilimrandom-}
$\nu_-(P_n)\dto \nu_-$, for some random $\nu_-\in\sP(\oi)$.
  \item \label{tsemilimrandom+}
$\nu_+(P_n)\dto \nu_+$, for some random $\nu_+\in\sP(\oi)$.
  \end{romenumerate}
If these hold, then $\nu_\pm\eqd\nu_\pm(\Pi)$ and
$\Pi\eqd \Pi_{{\nu_-}}$; in particular, $\Pi$ is a.s.\ a
semiorder limit. 

As a special case, the result holds with non-random $\Pi$, $\nu_-$, $\nu_+$
and $\dto$ replaced by $\pto$.
\end{theorem}

\begin{proof}
\ref{tsemilimrandom}  $\implies$\ref{tsemilimrandom-},\ref{tsemilimrandom+}.
By \refL{Lnu+-}, which also shows
$\nu_\pm\eqd\nu_\pm(\Pi)$.

\ref{tsemilimrandom-}  $\implies$\ref{tsemilimrandom}.  
Since $\cpq$ is a compact metric space, the space $\sP(\cpq)$ of
distributions on $\cpq$ is compact (see \eg{} \cite{Billingsley}); thus we
can select a subsequence along which $P_n$ converges in distribution:
$P_n\dto\Pi$ for some random $\Pi\in\cpq$. It follows that a.s.\
$\Pi\in\cpoo$. Moreover, if $Q=\sH$ or $\sL$, 
then, along the subsequence,
$0=\tind(Q,P_n)\dto\tind(Q,\Pi)$ so $\tind(Q,\Pi)=0$ a.s.;
hence $\Pi\in\SOoo$ a.s.\ by \refT{TsemiorderHL}.
Furthermore, it follows from \refL{Lnu+-} that,
still along the subsequence,
$\nu_-(P_n)\dto\nu_-(\Pi)$; thus $\nu_-(\Pi)\eqd\nu_-$.

If $\Pi'$ is the limit in distribution of $P_n$ along some other
subsequence, we thus have $\nu_-(\Pi)\eqd\nu_-(\Pi')$.
Since $\nu_-$ is a homeomorphism by \refC{Csemi-}, it follows that
$\Pi\eqd\Pi'$. Hence every convergent subsequence of the distributions
$\cL(P_n)$ has the same limit, and it follows (using compactness again) that
the full sequence converges, \ie, \ref{tsemilimrandom} holds.

\ref{tsemilimrandom+}  $\implies$\ref{tsemilimrandom}.  By
\refR{Rreflection}.

Finally, \refC{Csemi-} yields
$\Pi= \Pi_{{\nu_-(\Pi)}}\eqd \Pi_{{\nu_-}}$.
\end{proof}

\begin{example}
The random graph order $G(n,p)$ is obtained by 
regarding $G(n,p)$ as a random directed graph, with all edges directed $i\to
j$ when $i<j$, and taking the transitive closure, see \eg{} \cite{PittelT}.

These random orders are (typically) not semiorders, but \citet{BrightwellG}
have shown that their limits are semiorder limits: if
\ntoo{} and
$p\qw\log p\qw/n\to c$ for some $c\in \oi$, then 
$G(n,p)\pto \Pi_c$, where the poset limit $\Pi_c$ is represented by the
kernel $W_c(x,y)\=\ett{x+c<y}$ on $(\oi,\gl)$; thus $W_c=W_{g_c}$ for the
function $g_c(x)\=\min(x+c,1)$ with $g_c\in\cgrc$, and thus
$\Pi_c=\Pi_{g_c}\in\SOoo$. 
(More precisely, it suffices that
$\min(p\qw\log p\qw/n,1)\to c$, and this exhausts all possible poset limits,
see \cite{BrightwellG}.
In \cite{BrightwellG},
the limit is 
described as the \ops{} $\sss_{g_c}$, see \refR{RSg};
it is there denoted  $S_c$.)

We have $F_-(\Pi_c)=g_c$ and by an easy calculation, or by \refR{RF+-}, 
$F_+(\Pi_c)=g_c$ too. $\nu_-(\Pi_c)=\nu_+(\Pi_c)$ is the distribution of the
random variable $\max(U-c,0)$ with $U$ uniformly distributed on $\oi$.

See also \cite[Example 9.4]{SJ224}, where the kernel $W_c$ is denoted
$W_{c\qw}$.
\end{example}

\begin{example}
\citet{BrightwellG} have shown, more generally, 
that if a sequence of classical sequential
growth models has a poset limit (in probability, see \cite[Remark 4.3]{SJ224}), 
then the limit is a semiorder limit.
(In \cite{BrightwellG}  called an almost-semiorder, see
\refE{Ealmostsemiorder}.) 
In fact, they prove that if $P_n$ is a random poset of order $n$ given by 
some classical sequential growth model (possibly different for different
$n$), then $\E\tind(\sH,P_n)\to0$ and $\E\tind(\sL,P_n)\to0$; if further
$P_n\pto \Pi$, then thus $\tind(\sH,\Pi)=\tind(\sL,\Pi)=0$ by dominated
convergence and $\Pi$ is a semiorder limit by \refT{TsemiorderHL}.

Although the orders $P_n$ are not semiorders (typically), the proof of
\refT{Tsemilimrandom} still 
holds by the result just mentioned. (It suffices that any subsequence limit
is a semiorder limit.) Thus, for example, $P_n\pto\Pi$ for a (non-random)
semiorder limit $\Pi$ if and only if $\nu_-(P_n)\pto\nu_-$ for some 
(non-random) distribution $\nu_-$,
and then $\Pi=\Pi_{\nu_-}$.

\citet{BrightwellG} gave sufficient conditions for convergence, which are
related to the result just mentioned but more explicit.
Under these conditions,
they also gave a representation of the limit 
as a semiordered space which they denote by $T_r$, 
where $r:\oi\to[0,\infty]$ is some Borel function; 
in our notation the limit
is $\Pi_{T_r}$, and it is defined 
by the kernel
$\tW_r(x,y)\=\ett{\int_x^y r(t)\,dt>1}$ on 
$(\oi,\gl,<)$.
This is closely related to the representation given in \refT{Tsemi}; in fact, 
$\tW_r=W_g$ where $g(x)\=\sup\bigset{y\le1:\int_x^yr(t)\dd t\le1}$
and ${T_r}=\sss_g$ defined in \refR{RSg}.

It is easily seen that $g(x)=1/2$ for $x<1/4$, $g(x)=\min(x+1/2,1)$
for $x\ge1/4$ cannot be represented as $\tW_r$, so not every semiorder limit
can be represented as some $T_r$.
\end{example}

\section{Relations to interval graph limits}\label{Sgraphs}

Given a poset $P$, let $\Psi(P)$ denote its \emph{comparability graph}, \ie,
the graph with vertex set $P$ and edge set \set{ij:i<_P j\text{ or } j<_p i};
further, let $\bPsi(P)$ denote the complement of $\Psi(P)$.
Thus $\Psi$ and $\bPsi$ are maps from the set $\cP$ of (unlabelled) finite
posets to the set $\cU$ of (unlabelled) finite graphs; 
the following lemma says that these maps are
continuous in the topologies used for poset limits and graph limits.
As in \cite{SJ209}, let
$\cuq$ be the completion of $\cU$ and let
$\cuoo\=\cuq\setminus\cU$ be the set of graph limits.

\begin{lemma}
The maps $\Psi$ and $\bPsi$ extend to continuous maps 
$\cpq\to\cuq$, mapping
the space $\cpoo$ of posets limits into the space $\cuoo$ of graph limits.
In particular, if a sequence 
of posets $P_n$ converges, 
then so do the sequences $\Psi(P_n)$ and $\bPsi(P_n)$ of graphs,
and if $P_n\to\Pi$, then $\Psi(P_n)\to\Psi(\Pi)$, $\bPsi(P_n)\to\bPsi(\Pi)$.  
\end{lemma}
\begin{proof}
  If $F$ is a finite graph, then $\tind(F,\Psi(P))=\sum_i \tind(F_i,P)$, where
the $t$ on the left-hand side is the graph version of the  functional $t$, and
  $F_1,F_2,\dots$ are the different digraphs obtained by directing the edges
  in $F$ (it suffices to consider such digraphs that are posets, since
  otherwise $\tind(F_i,P)=0$).
This shows that the map $\Psi$ is continuous, and extends
to $\cpq$, by  the definition of the
topologies in $\cuq$ and $\cpq$, see \cite{SJ209} and \cite{SJ224}.

Similarly, the map $G\mapsto\overline G$ mapping a graph to its complement
extends to a continuous map $\cuq\to\cuq$ since $\tind(F,\overline
G)=\tind(\overline F,G)$. Thus the composition $\bPsi$ too is continuous.

Since $|\Psi(P)|=|\bPsi(P)|=|P|$ for any finite poset $P$, it follows easily
that $\Psi$ and $\bPsi$ map $\cpoo$ into $\cuoo$.
\end{proof}

\subsection{Interval orders}

If $P$ is an interval order, then $\bPsi(P)$ is an interval graph.
Conversely, every interval graph can be obtained in this way. Thus, 
$\bPsi$ maps $\IO$ onto the set $\IG$ of interval graphs.
Let $\bIG$ be the closure of $\IG$ in $\cuq$ and let
$\IGoo\=\bIG\setminus\IG$ be the set of interval graph limits.

\begin{theorem}\label{TIOIG}
  $\bPsi$ maps\/ $\bIO$ onto $\bIG$ and\/ $\IOoo$ onto $\IGoo$. 
Moreover, $\bPsi\qw(\IGoo)=\IOoo$, \ie, if
  $\Pi\in\cpq$ and $\bPsi(\Pi)$ is an interval graph limit, then $\Pi$ is an
  interval order limit.
\end{theorem}

\begin{proof}
  Since $\bPsi:\IO\to\IG$, by continuity $\bPsi:\bIO\to\bIG$. Moreover,
the range is compact, since $\bIO$ is compact and $\bPsi$ is continuous, and
also 
dense in $\bIG$ since it contains $\bPsi(\IO)=\IG$; hence the range 
$\bPsi(\bIO)=\bIG$.
Since $\bPsi(\IO)\subseteq\IG$ and $\bPsi(\IOoo)\subseteq\IGoo$,
it follows that $\bPsi(\IOoo)=\IGoo$.

If $\Pi\in\cpq$ and $\bPsi(\Pi)\in\IGoo$, then
$\Pi\in\cpoo$. Furthermore, as an undirected graph,
$\sH=\bsC_4$, and thus
\begin{equation*}
\tind(\sH,\Pi)\le \tind(\bsC_4,\Psi(\Pi))=\tind(\sC_4,\bPsi(\Pi))=0.  
\end{equation*}
Hence, $\Pi\in\IOoo$ by \refT{TintervalH}.
\end{proof}

Using the representations of interval order limits and interval graph limits
by measures $\mu\in\sP(\sssi)$ in \refT{Tinterval} and \cite{SJ254}, the map
$\bPsi:\IOoo\to\IGoo$ is given by $\bPsi(\Pi_\mu)=\gG_\mu$, 
since $\bPsi(P(n,\Pi_\mu))=G(n,\gG_\mu)$ 
as is seen from the definitions of the random posets $P(n,\Pi_\mu)$ and the
random graphs $G(n,\gG_\mu)$.

The surjection $\bPsi:\IO\to\IG$ is not a bijection. If $G$ is a labelled
interval graph, and $G$ is not complete,
then there are always at least two interval orders on the
vertex set that yield the graph $G$, since we may reverse any such
order;
these orders are sometimes, but not always equivalent as unlabelled
posets.
(For example, if $G$ is an empty graph on $[n]$, then any total order of
$[n]$ will do, but they all yield the same unlabelled poset.)
A simple example of non-uniqueness is by taking $G$ to be the disjoint sum
$\sK_1+\sK_2+\sK_3$; then $\bPsi\qw(G)$ consists of the 6 (weak) orders
obtained by taking the three vertex sets of the subgraphs
$\sK_1,\sK_2,\sK_3$ in any order, letting the vertices inside each set be
incomparable to each other. 
See further \cite[Section 3.6]{Fishburn} (labelled case) and
\cite{Hanlon} (unlabelled case); these references contain among other
results characterizations of interval graphs $G$ such that $\bPsi\qw(G)$
contains only one poset, or two poset with opposite orders.

\begin{problem}\label{Pinterval}
Describe $\bPsi\qw(\gG)$ for a general interval graph limit $\gG$.
In particular, 
  characterize the interval graph limits $\gG$ such that 
$\bPsi\qw(\gG)$ consists of a single poset limit, 
or  two poset limits  related by reflection.
\end{problem}

It seems likely that a solution to this problem could be combined with the
unique representation in \refT{Tunique} to yield a unique representation
of interval graph limits by some class of kernels.

\subsection{Semiorders}

$P$ is a semiorder if and only if $\bPsi(P)$ is a unit interval graph (a.k.a.\
\emph{indifference graph}); moreover, every unit interval graph is $\bPsi(P)$
for some semiorder $P$
\cite[Theorem 3.2]{Fishburn}.
Thus, $\bPsi$ maps $\SO$ onto the set $\UIG$ of unit interval graphs.
Let $\bUIG$ be the closure of $\UIG$ in $\cuq$ and let
$\UIGoo\=\bUIG\setminus\UIG$ be the set of unit interval graph limits.

\begin{theorem}\label{TSOUIG}
  $\bPsi$ maps $\bSO$ onto $\bUIG$ and $\SOoo$ onto $\UIGoo$. 
Moreover, $\bPsi\qw(\UIGoo)=\SOoo$, \ie, if\/
$\Pi\in\cpq$ and\/ $\bPsi(\Pi)$ is a unit interval graph limit, then $\Pi$ is a
semiorder limit.
\end{theorem}
\begin{proof}
As the proof of \refT{TIOIG},
\emph{mutatis mutandis}; 
now using \refT{TsemiorderHL}. 
\end{proof}

 \refP{Pinterval} seems considerably easier for unit interval graph limits in
view of the results for unit interval graphs in \cite{Hanlon,Fishburn}. 
In particular, if $G$ is a connected unit interval graph, then $\bPsi\qw(G)$
consists of just one or two posets $P$ and $P\refl$ (which may be the same
or not, as unlabelled posets) \cite[Theorem 3.10]{Fishburn}. 
This leads to the following conjecture.

\begin{conjecture}
If\/ $\gG$ is a connected unit interval graph limit, then $\bPsi\qw(\gG)$
consists of either one semiorder limit $\Pi$ (with $\Pi\refl=\Pi$) or two
semiorder limits $\Pi$ and $\Pi\refl$. 
\end{conjecture}

If this conjecture would be proven, it would by \refT{Tsemi} lead to a result on
characterization of
representations of unit interval graph limits, and perhaps to a way of
selecting unique or almost unique such representations.

\newcommand\AAP{\emph{Adv. Appl. Probab.} }
\newcommand\JAP{\emph{J. Appl. Probab.} }
\newcommand\JAMS{\emph{J. \AMS} }
\newcommand\MAMS{\emph{Memoirs \AMS} }
\newcommand\PAMS{\emph{Proc. \AMS} }
\newcommand\TAMS{\emph{Trans. \AMS} }
\newcommand\AnnMS{\emph{Ann. Math. Statist.} }
\newcommand\AnnPr{\emph{Ann. Probab.} }
\newcommand\CPC{\emph{Combin. Probab. Comput.} }
\newcommand\JMAA{\emph{J. Math. Anal. Appl.} }
\newcommand\RSA{\emph{Random Struct. Alg.} }
\newcommand\ZW{\emph{Z. Wahrsch. Verw. Gebiete} }
\newcommand\DMTCS{\jour{Discr. Math. Theor. Comput. Sci.} }

\newcommand\AMS{Amer. Math. Soc.}
\newcommand\Springer{Springer}
\newcommand\Wiley{Wiley}

\newcommand\vol{\textbf}
\newcommand\jour{\emph}
\newcommand\book{\emph}
\newcommand\inbook{\emph}
\def\no#1#2,{\unskip#2, no. #1,} 
\newcommand\toappear{\unskip, to appear}

\newcommand\webcite[1]{
\texttt{\def~{{\tiny$\sim$}}#1}\hfill\hfill}
\newcommand\webcitesvante{\webcite{http://www.math.uu.se/~svante/papers/}}
\newcommand\arxiv[1]{\webcite{arXiv:#1.}}

\def\nobibitem#1\par{}

\end{document}